\documentclass[12pt]{article}

\usepackage[letterpaper]{geometry}
\usepackage{amsmath,amsthm,geometry,tikz,tikz-cd,enumitem,amssymb}
\usetikzlibrary{matrix,decorations.pathmorphing,arrows}
\usetikzlibrary{decorations.markings}
\tikzset{commutative diagrams/.cd,arrow style=tikz,diagrams={>=stealth'}}
\tikzset{->-/.style={decoration={
  markings,
  mark=at position #1 with {\arrow{>}}},postaction={decorate}}}
 \tikzset{-<-/.style={decoration={
  markings,
  mark=at position #1 with {\arrow{<}}},postaction={decorate}}}
 \tikzset{>=stealth}

\usepackage{caption}
\usepackage{subcaption}

\usepackage{lipsum}

\newcommand\blfootnote[1]{%
  \begingroup
  \renewcommand\thefootnote{}\footnote{#1}%
  \addtocounter{footnote}{-1}%
  \endgroup
}

\newcommand{\mr}{\mathring}
\newcommand{\Z}{\mathbb{Z}}

\DeclareMathOperator{\closure}{cl}

\DeclareMathOperator{\cone}{cone}

\DeclareMathOperator{\intr}{int}

\renewcommand{\phi}{\varphi}

\numberwithin{equation}{section}
\newtheorem{theorem}[equation]{Theorem}

\newtheorem{lemma}[equation]{Lemma}
\newtheorem{corollary}[equation]{Corollary}

\newtheorem*{mainthm}{Theorem \ref{THBS from cusp condition}}
\newtheorem*{maincor}{Corollary \ref{main corollary}}

\theoremstyle{definition}

\newtheorem{question}[equation]{Question}

\theoremstyle{remark}

\begin{document}

\title {Taut branched surfaces from veering triangulations}
\author {Michael Landry} 
\date{}

\maketitle

\begin{abstract}
Let $M$ be a closed hyperbolic 3-manifold with a fibered face $\sigma$ of the unit ball of the Thurston norm on $H_2(M)$. If $M$ satisfies a certain condition related to Agol's veering triangulations, we construct a taut branched surface in $M$ spanning $\sigma$. This partially answers a 1986 question of Oertel, and extends an earlier partial answer due to Mosher.
\end{abstract}

\blfootnote{This material is based upon work supported by the National Science Foundation Graduate Research Fellowship Program under Grant No. DGE-1122492, and by the National Science Foundation under Grant No. DMS-1610827 (PI Yair Minsky). Any opinions, findings, and conclusions or recommendations expressed in this material are those of the author and do not necessarily reflect the views of the National Science Foundation. }

\section{Introduction}
Let $M$ be a closed, irreducible, atoroidal 3-manifold. In \cite{Oe}, Oertel shows that each closed face $\sigma$ of the Thurston norm ball of $M$ possesses a finite collection of taut oriented branched surfaces which together carry representatives of all integral classes in $\cone(\sigma)$. We say that these branched surfaces \textbf{span} $\sigma$. One question he asked and left open was whether this collection could have size 1. That is, when is it possible to find a single taut oriented branched surface which spans $\sigma$? Our main result is a partial answer to this question when $M$ is hyperbolic and $\sigma$ is fibered.

\begin{mainthm}[Main Theorem]
If every boundary torus of $\mathring{M}$ witnesses at most two ladderpole vertex classes of $H_2(M)$, then $M$ has a taut homology branched surface spanning $\sigma$.
\end{mainthm}

The definition of $\mr M$, and that of \emph{ladderpole vertex class}, which involves Agol's veering triangulation, can be found in Sections \ref{background} and \ref{boundary train track stuff} respectively. 

To quickly summarize the previous study of Oertel's question: in \cite{SB}, it is shown by a counterexample that the answer is ``not always." However, in \cite{Mo} Mosher proves that in the case when $M$ is hyperbolic and $\sigma$ is fibered, there exists such a branched surface if the vertices of $\sigma$ have positive intersection with the singular orbits of the Fried suspension flow $\phi$ corresponding to $\sigma$.

The main theorem \ref{THBS from cusp condition} strengthens Mosher's result because ladderpole classes in particular have intersection 0 with some singular orbits of the Fried flow $\phi$. While the condition in the statement of the main theorem may seem mysterious, we prove in Theorem \ref{small betti} that it holds, in particular, when $H_2(M)$ has rank at most 3. This gives us the following corollary.

\begin{maincor}
If $b_2(M)\le 3$, then $M$ has a taut homology branched surface spanning $\sigma$.
\end{maincor}

The technique we employ to improve Mosher's result uses Agol's veering triangulation of a pseudo-Anosov mapping torus $M'$ which is missing the singular orbits of Fried's flow $\phi$. One nice property of this veering triangulation is that its 2-skeleton has the structure of a taut oriented branched surface which spans a fibered face $\sigma'$ of $M'$. 

We work in the cusped manifold $M'$ and make use of the veering triangulation before moving back to $M$ by Dehn filling. By understanding how the veering triangulation sits in $M$, we can construct a face-spanning taut oriented branched surface in $M$ as long as our  condition regarding ladderpole vertex classes is met.

Veering triangulations were introduced in \cite{Ag}, where they were used to provide an alternative proof of the theorem in \cite{FLM} stating that the mapping tori of small-dilatation pseudo-Anosovs come from Dehn filling on finitely many cusped hyperbolic manifolds. 
In \cite{HRST} it is shown that veering triangulations admit positive angle structures, and in \cite{FG} the authors give a lower bound on the smallest angle in such a triangulation in terms of combinatorial data coming from the veering triangulation. In light of these results, one could hope that veering triangulations might be realized geometrically. However, in \cite{HIS} there is an explicit example of a non-geometric veering triangulation.
Gu\'eritaud proved that the veering triangulation controls the combinatorics of the Cannon-Thurston map associated to the fully punctured manifold $M'$ in \cite{Guer}. Most recently, Minsky and Taylor found connections between subsurface projections and the veering triangulation in \cite{MinTay}.

\textbf{Acknowledgements.} I would like to thank the anonymous referee for her or his helpful comments and suggestions. Thanks are also due to Siddhi Krishna, who identified an expositional error in an earlier version of the paper. Finally, I would like to thank my advisor Yair Minsky for many helpful conversations during this research.

\section{Background}\label{background}

\subsection{Notation} \label{notation}

Throughout this document, certain notations will remain fixed. All homology groups are assumed to have coefficients in $\mathbb{R}$ unless otherwise stated.
We consider a closed pseudo-Anosov mapping torus $M$, and a closed face $\sigma$ of the unit ball of the Thurston norm $x$ on $H_2(M)$. There is a pseudo-Anosov flow $\phi$ on $M$ which Fried showed in \cite{Fr79} is associated to $\sigma$ in a natural way. We denote the union of the singular orbits of $\phi$ by $c$ and define $\mr M:=M\setminus U$ where $U$ is an open tubular neighborhood of $c$. The boundary of $\mr M$, $\partial \mr M$, is a union of tori. There is a natural embedding $P\colon H_2(M)\to H_2(\mr M,\partial \mr M)$ 
which comes from the exact sequence $0=H_2(U)\to H_2(M)\to H_2(M,U)\to\cdots$ and the excision isomorphism $H_2(\mr M,\partial \mr M)\cong H_2(M,U)$. This is induced at the level of chains by sending a chain $S$ to $S\setminus U$. We use the notation $P(\alpha)=\mr \alpha$. More exposition of some of these ideas will be provided in the following subsections.

\subsection{Thurston norm} \label{Thurston norm section}

We begin by reviewing the definition of, and some salient facts about, the Thurston norm. For a more detailed treatment, see Thurston's original paper \cite{Th1} or Candel and Conlon's textbook \cite{CC}.

Let $Z$ be an orientable three-manifold. The Thurston norm is a seminorm on $H_2(Z,\partial Z)$ defined as follows. If $ S $ is a connected surface embedded in $ Z $, define $\chi _- (S): = \max \{ 0,-\chi  (S)\} $, where $\chi  $ denotes Euler characteristic. If $ S $ has multiple connected components, then $\chi _- (S) :=\sum^{}_{i }\chi _-(S_i) $ where the sum is over the components $S_i$ of $ S $.

If $\alpha\in H_2(Z,\partial Z)$ is an integral class, we can always find an embedded surface representing $ \alpha $.  We define 
\[
x (\alpha ) = \min _{[S]=\alpha}\chi_-(S),
\]
where the minimum is taken over embedded representatives of $\alpha  $, and call $x(\alpha)$ the \textbf{Thurston norm} of $\alpha  $. 

The Thurston norm is homogeneous and satisfies the triangle inequality for integral classes. It can be extended to all rational classes in $H_2(Z,\partial Z)$ by homogeneity, whereafter it extends uniquely to a seminorm on all of $H_2(Z,\partial Z)$. If $Z$ is incompressible, atoroidal, anannular, and boundary-incompressible (i.e. has no essential surfaces of nonnegative Euler characteristic), then $x$ is a bona fide norm on $H_2(Z,\partial Z)$, and not just a seminorm. In this paper we will deal only with pseudo-Anosov mapping tori, so we will assume $x$ is a norm from this point forward.

From the fact that $ x $ takes integer values on the integral lattice, one can show that the unit sphere of $x$ is a finite-sided convex polyhedron. Moreover, this polyhedron encodes information about how $Z$ fibers over $ S ^ 1 $. If $F$ is a fiber of some fiber bundle $ F\rightarrow Z\rightarrow  S ^ 1 $, then $ [F] $ lies in the interior of the positive cone over a top-dimensional face. If $\alpha$ is any integral class in the interior of the same cone, then $\alpha$ is represented by a fiber of some fiber bundle over the circle. The closed top-dimensional face associated to this cone is called a \textbf{fibered face}.

In the setting of this paper, a closed hyperbolic manifold $M$ with fibered face $\sigma$,  David Fried proved in \cite{Fr79} that we can associate to $\sigma $ a flow $\varphi $ with very nice properties.
Namely, any primitive integral class $\alpha\in \intr \cone(\sigma) $, where $\intr$ denotes interior, is represented by a cross-section $S$ to $\varphi $, and the first return map of $S$ is pseudo-Anosov. Hence we can think of $\phi$ as the simultaneous suspension flow of all monodromies of fibers corresponding to $\sigma$, and the singular orbits of $\phi$ as the suspensions of the singular points of those monodromies.

\subsection{Taut branched surfaces}\label{branched surfaces section}

A \textbf{branched surface} $B$ is a smooth codimension-1 object in a 3-manifold, analogous to a train track in a surface, which organizes the data of various embedded surfaces. Locally, a branched surface looks like a stack of disks   $D_1, \dots, D_n$ such that $ D_i $ is glued to $ D_{i +1} $ for  $ i< n $ along the closure of one component of the complement of a smooth arc through $D_i$. The union of the images of smooth curves is called the \textbf{branching locus}. The smooth structure of $B$ is such that the inclusion of each $ D_i $ is smooth (\cite{Oe88}). The \textbf{sectors} of $ B $ are the connected components of the complement of the branching locus; these are analogous to the edges of a train track.

A regular neighborhood $N(B)$ of a branched surface $B$ can be foliated by line segments transverse to $B$. This foliation is called the \textbf{vertical foliation} of $ N (B) $. If it is possible to consistently orient the leaves of the vertical foliation, then $ B $ is called an \textbf{oriented branched surface}.

We say that $ B $ \textbf{carries} $ S $ if $ S $ is embedded in $ N (B) $ transverse to the vertical foliation; this is analogous to a train track carrying a curve. In the same way that a train track inherits nonnegative integer edge weights from a carried curve, a surface $S$ carried by $ B $ assigns a weight to each sector of $ B $. If these weights are all positive, we say $B$ \textbf{fully carries} $S$. 

Conversely, a collection of nonnegative integral weights on sectors of $B$ which satisfy the linear equations determined by the branching determines an isotopy class of surface carried by $B$. In fact, any real weights satisfying the branching equations naturally determine a homology class; we say that the homology classes corresponding to nonnegative weights  are \textbf{carried by} $B$.

If we allow negative weights, there is a natural vector space whose elements are collections of real weights satisfying the branching equations. In this vector space, the integral points in the cone of non-negative weights correspond to surfaces carried by $B$.

Branched surfaces are interesting in part because of the surfaces they carry, so it is natural to distinguish types of branched surfaces which carry surfaces with nice properties.
If a branched surface carries only surfaces realizing the minimal $\chi_-$ in their homology class, we say it is \textbf{almost taut}. Following Oertel, we say $B$ is \textbf{taut} if it carries only \emph{incompressible} surfaces which attain the minimal $\chi_-$ in their homology class. We say $B$ is a \textbf{homology branched surface} if $B$ is oriented and for each point $p \in B $, there exists a closed oriented transversal through $ p $, i.e. an oriented loop through $p$ whose intersection with $N(B)$ consists of leaves of the vertical foliation with the correct orientation. Since any surface carried by a homology branched surface has nonzero algebraic intersection number with a closed curve, homology branched surfaces carry only homologically nontrivial surfaces.

The following lemma, which we will use later, is probably known to many. We record a proof here for convenience.

\begin{lemma}[]\label{closed cone}
Let $Z$ be a compact 3-manifold, and $B\subset Z$ a homology branched surface. Then the cone of classes carried by $B$ surface is closed.
\end{lemma}

\begin{proof}
Consider the natural linear map $L$ from the vector space $ V $ of weights satisfying the branching equations on $B$ to $H_2(Z,\partial Z)$. Let $A\subset V$ be the closed cone of nonnegative weights on $B$; our goal is to show that $L(A)$ is closed.

Choose a norm $|| \cdot ||$ on $V$, and let $A^1=\{v\in A\mid ||v||=1\}$. Then $A^1$ is compact, and hence $L(A^1)$ is compact. Moreover, $0\notin L(A^1)$ because $B$ is a homology branched surface.

Note that $L(A)=\cone(L(A^1))$. Since the cone over a compact set not containing $0$ is closed, we are done.
\end{proof}

Note that an almost taut branched surface is not necessarily taut. For example, if $ B $ carries the boundary of a solid torus, $B$ is not taut. However, since this torus realizes the Thurston norm of the homology class 0, carrying the torus does not preclude $ B $ from being almost taut. 

If an almost taut homology branched surface $B$ lies in a pseudo-Anosov mapping torus, it is taut. In general, we have the following.

\begin{lemma}\label{almost taut to taut}
Let $N$ be a manifold such that the Thurston norm $x$ is a norm (i.e. not just a seminorm) on $H_2(N,\partial N)$, and let $B\subset N$ be an almost taut branched surface. If $B$ is also a homology branched surface, then $B$ is taut.
\end{lemma}

\begin{proof}
Suppose $S$ is a compressible surface carried by $B$, and let $S'$ be the surface obtained by compressing $S$ along a compression disk. Since $B$ is almost taut, $\chi_-(S)=\chi_-(S')$, and since $\chi(S')=\chi(S)+2$, $S$ must be a torus, annulus, disk, or sphere. Since $B$ is a homology branched surface, $S$ is homologically nontrivial, a contradiction. 
\end{proof}

In the course of proving Theorem 4 in \cite{Oe}, Oertel proves the following useful criterion for almost tautness. Since he doesn't state the result explicitly, we record a proof here for the reader's convenience.

\begin{lemma}[\cite{Oe}]\label{tautness from carrying}
Let $N$ be as above. Suppose $ B\subset M$ is an oriented branched surface which fully carries a minimal-$\chi_-$ representative $\Sigma $ of a single class in $H _2(N,\partial  N) $. If $B$ does not carry any spheres or disks, then $ B $ is almost  taut.
\end{lemma}

\begin{proof}
Suppose for a contradiction that $ B $ carries a surface $ S $ which is homologous to $ S' $ with $\chi _-(S) >\chi _-(S') $. Since $\Sigma $ is fully carried, there exists $ n\in \mathbb{N}$ such that $n\Sigma $ (i.e. $ n $ parallel copies of $\Sigma  $) is carried with greater weights than $ S $ on each sector of $ B $. We have 
\[
n\Sigma = (n\Sigma -S) + S,
\]
where $n\Sigma-S$ denotes the surface arising from subtracting the weights corresponding to $S$ from the weights corresponding to $n\Sigma$, and $+$ denotes oriented sum in a regular neighborhood of $B$. Let $ F := n\Sigma - S$. 
Now
\begin{align*}
\chi _- (n\Sigma ) & =\chi _- (F) +\chi _- (S)&\text{(because $B$ does not carry disks or spheres)}\\
& >\chi _- (F) +\chi _- (S')&\\
&\ge x([F])+x([S'])&\\
&\ge x([n\Sigma]).&
\end{align*}
The last inequality is the triangle inequality for $x$. This is a contradiction since $n\Sigma$ realizes the minimal $\chi_-$ in $[n\Sigma]$.
\end{proof}

\subsection{Branched surfaces and faces of the Thurston norm ball}

If two homologically nontrivial surfaces $S$ and $T$ are carried by a taut oriented branched surface $B$, we can perform an oriented cut-and-paste along the branching locus of $B$ to form a surface $S+T$ representing $[S]+[T]$ and carried by $B$. Since $B$ is taut, $S+T$ is norm-minimizing. Also, $\chi_-(S+T)=\chi_-(S)+\chi_-(T)$, as none of $S$, $T$, or $S+T$ has any sphere or disk components. This is because they are all carried by a taut branched surface. Thus
\[
x([S]+[T])=x([S+T])=\chi_-(S+T)=\chi_-(S)+\chi_-(T)=x([S])+x([T]).
\] 

The faces of the Thurston norm unit ball are projectivizations of the maximal cones on which $x$ is linear, so we conclude that $[S]$ and $[T]$ lie in $\cone(\sigma)$ for some face $\sigma$ of the Thurston norm ball. Oertel observed this and asked the following question.

\begin{question}[\cite{Oe}]\label{Oertel's question}
Let $M$ be a simple (compact, irreducible, atoroidal) 3-manifold. For each face of the unit ball of the Thurston norm on $ H_2 (M,\partial  M) $, is it possible to find a taut oriented branched surface which carries a norm-minimizing representative of every projective homology class in $\cone(\sigma)$? 
\end{question}

In \cite{SB}, a closed 3-manifold is constructed for which the answer to Question \ref{Oertel's question} is No. More specifically, Sterba-Boatwright produces a face of this manifold's Thurston norm ball with the following property: any branched surface carrying norm-minimizing representatives of all classes in that face also carries a compressible torus, so it cannot be taut. Hence the answer to Oertel's question is not an unqualified Yes.
However, in \cite{Mo}, Question \ref{Oertel's question} is answered in the affirmative for a fibered face of a closed pseudo-Anosov mapping torus in the case that each integral class in the cone over the face has positive intersection number with the singular orbits of the suspension flow for that face. 

In this paper we extend Mosher's result by using the relatively new technology of veering triangulations, which we describe now.

\subsection{Veering Triangulations}
A \textbf{taut ideal tetrahedron} is an ideal tetrahedron with the following extra data: each edge is labelled 0 or $\pi  $ so that the sum of the labels of edges incident to any ideal vertex is $\pi$, each face is co\"oriented so that two faces point out and two faces point in, and face co\"orientations change only along edges labeled 0. Note that the word ``taut" is used here in a different sense than when it modifies ``branched surface." A \textbf{taut ideal triangulation} of a 3-manifold is an ideal triangulation by taut tetrahedra such that face co\"orientations agree, and for each edge $ e $, the sum of $ e $'s labels over all incident tetrahedra is $ 2\pi  $. The 2-skeleton of such a triangulation $\triangle $ can be pinched along each edge to give a branched surface $B_\triangle $ in the manifold, as seen in Figure \ref{lookingatT}. In this way we think of the edge labels of a taut tetrahedron as angles. Taut ideal triangulations were introduced in \cite{La}.

The condition that makes a triangulation veering is simple to draw but more complicated to define.

Consider a taut ideal triangulation of an oriented 3-manifold $L$. Up to combinatorial equivalence there is only one taut ideal tetrahedron. However, if we distinguish a 0-edge there are 2 types embedded in $L$ up to oriented equivalence, as you can see in Figure \ref{taut_types}. The taut ideal triangulation is \textbf{veering} if for each edge $e$, 
(i) if we circle around $ e $ and read off edge angles, no two $\pi$'s are circularly adjacent, and 
(ii) each tetrahedron for which $ e $ is a 0-edge is of the same type when $ e $ is distinguished.
This situation is shown in Figure \ref{edge link}.
If $L$ is not orientable, a taut ideal triangulation of $L$ is veering if its lift to the orientation cover of $L$ is veering.

There is an alternative definition of veering, due to Gueritaud in \cite{Guer}, which says that a taut triangulation is veering if its 1-skeleton possesses a certain type of two-coloring. That definition is equivalent to the one given above, which we use throughout the paper.

\begin{figure}
\centering
\begin{tikzpicture}[scale=3]
\tikzstyle myBG=[line width=6pt,opacity=1.0]
\draw[line width=1pt](0,1)--(1,0);
\draw[white, myBG] (0,0)--(1,1);
\draw [line width=1pt](0,0)--(1,1);
\draw [line width=1pt](0,1)--(1,1)--(1,0)--(0,0);
\draw [line width=3.5pt] (0,0)--(0,1);
\foreach \x in {0,1}
\foreach \y in {0,1}
\node at (\x,\y) [circle, fill=white, scale=0.55]{};
\foreach \x in {0,1}
\foreach \y in {0,1}
\draw [line width=1pt](\x,\y) circle (.38mm);
\end{tikzpicture}
\quad\quad
\begin{tikzpicture}[scale=3]
\tikzstyle myBG=[line width=6pt,opacity=1.0]
\draw[line width=1pt](0,0)--(1,1);
\draw[white, myBG] (0,1)--(1,0);
\draw [line width=1pt](0,1)--(1,0);
\draw [line width=1pt](0,1)--(1,1)--(1,0)--(0,0);
\draw [line width=3.5pt] (0,0)--(0,1);
\foreach \x in {0,1}
\foreach \y in {0,1}
\node at (\x,\y) [circle, fill=white, scale=0.55]{};
\foreach \x in {0,1}
\foreach \y in {0,1}
\draw [line width=1pt](\x,\y) circle (.38mm);
\end{tikzpicture}
\caption{The two types of taut tetrahedron with distinguished (bold) 0-edge. In this drawing the co\"orientation is pointing out of the paper, the vertical and horizontal edges are 0-edges, and the diagonal edges are $\pi$-edges.}
\label{taut_types}
\end{figure}
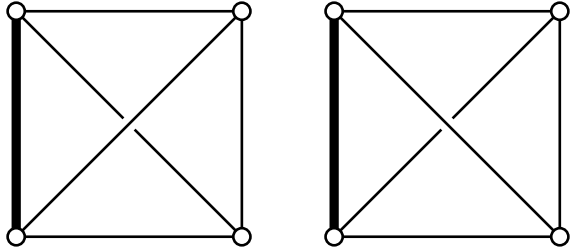

\begin{figure}
\centering
\begin{tikzpicture}[scale=2.7]
\tikzstyle myBG=[line width=5pt,opacity=1.0]
\draw [line width=1pt] (.15,1.4)--(-1.25,1.3);
\draw [line width=1pt] (.15,1.4)--(-1.3,.4);
\draw [line width=1pt] (.15,1.4)--(-.7,-.4);
\draw [line width=1pt] (0,0)--(1.55,.3);
\draw [line width=1pt] (0,0)--(1.4,1.6);
\draw [white, myBG] (0,0)--(-.8,1.8);
\draw [white, myBG] (0,0)--(-1.25,1.3);
\draw [white, myBG] (0,0)--(-1.3,.4);
\draw [white, myBG] (.15,1.4)--(1.1,-.35);
\draw [white, myBG] (.15,1.4)--(1.55,.3);
\draw[line width=1pt] (0,0)--(1.1,-.35)--(1.55,.3)--(1.4,1.6)--(.15,1.4)
--(-.8,1.8)--(-1.25,1.3)--(-1.3,.4)--(-.7,-.4)--(0,0);
\draw [line width=1pt] (0,0)--(-.8,1.8);
\draw [line width=1pt] (0,0)--(-1.25,1.3);
\draw [line width=1pt] (0,0)--(-1.3,.4);
\draw [line width=1pt] (.15,1.4)--(1.1,-.35);
\draw [line width=1pt] (.15,1.4)--(1.55,.3);
\draw[line width=3pt] (0,0)--(.15,1.4);
\node at (0,0) [circle, fill=white, scale=0.55]{};
\node at (1.1,-.35) [circle, fill=white, scale=0.55]{};
\node at (1.55,.3) [circle, fill=white, scale=0.55]{};
\node at (1.4,1.6) [circle, fill=white, scale=0.55]{};
\node at (.15,1.4) [circle, fill=white, scale=0.55]{};
\node at (-.8,1.8) [circle, fill=white, scale=0.55]{};
\node at (-1.25,1.3) [circle, fill=white, scale=0.55]{};
\node at (-1.3,.4) [circle, fill=white, scale=0.55]{};
\node at (-.7,-.4) [circle, fill=white, scale=0.55]{};

\draw [line width=1pt](0,0) circle (.38mm);
\draw [line width=1pt](1.1,-.35) circle (.38mm);
\draw [line width=1pt](1.55,.3) circle (.38mm);
\draw [line width=1pt](1.4,1.6) circle (.38mm);
\draw [line width=1pt](.15,1.4) circle (.38mm);
\draw [line width=1pt](-.8,1.8) circle (.38mm);
\draw [line width=1pt](-1.25,1.3) circle (.38mm);
\draw [line width=1pt](-1.3,.4) circle (.38mm);
\draw [line width=1pt](-.7,-.4) circle (.38mm);
\end{tikzpicture}

\caption{The bold edge satisfies the veering condition.}
\label{edge link}
\end{figure}
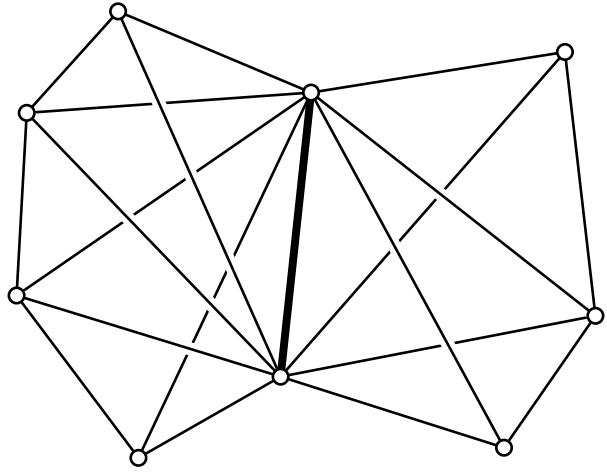

Recall that $c$ denotes the union of singular orbits of the Fried suspension flow in our pseudo-Anosov mapping torus $M$. In \cite{Ag}, Ian Agol shows how to construct a canonical  veering triangulation of $ M\setminus c$. He builds this triangulation using a sequence of ideal triangulations of the punctured surface which are dual to a periodic train track splitting sequence. The triangulations are related by Whitehead moves which determine the incidencies of taut tetrahedra. Taut ideal triangulations obtained from Whitehead moves in this way, a construction due to Lackenby in \cite{La}, are called \textbf{layered triangulations}.

Agol has proven the following theorem, which shows that his veering triangulation of $ M\setminus c$ is canonically associated to a fibered face of $M$. Gu\'eritaud provided an alternative proof, which is exposited in \cite{MinTay}, based on his alternative construction of the canonical veering triangulation in \cite{Guer}.

\begin{theorem}[Agol]\label{Agol's theorem}
The ideal layered veering triangulation $\tau$ of $ M\setminus c$ coming from a fibration associated to $\intr \cone(\sigma)$ is constant over $\intr \cone(\sigma)$. The 2-skeleton of this triangulation is a branched surface $B_\tau$ such that if $S$ is a fiber with $[S]\in \intr\cone(\sigma)$, some multiple of $ S\setminus c$ is fully carried by $B_\tau$.
\end{theorem}

We remark that $B_\tau$ is taut. Indeed, $B_\tau$ is transverse to $\phi$, the flow associated to $\sigma$, and a generic orbit of $\phi$ is dense in $M$; this allows us to find a closed transversal through every point of $B_\tau$, so $B_\tau$ is a homology branched surface. Since $M\setminus c$ is irreducible and boundary irreducible, this implies $B_\tau$ carries no disks or spheres. 
Indeed, any carried disk or sphere would be homologically nontrivial because $B_\tau$ is a homology branched surface, but this is ruled out by irreducibility and boundary irreducibility of $M\setminus c$.

The fact that $B_\tau$ fully carries a fiber (in fact, many fibers) of $M\setminus c$ means that $B_\tau$ is  almost taut by Lemma \ref{tautness from carrying} since fibers are norm-minimizing, see \cite{Th1}. By Lemma \ref{almost taut to taut}, $B_\tau$ is taut.

Since we want to use the veering triangulation $\tau$ to extract information about the Dehn filling $M$ of $M\setminus c$, it will be useful to consider the restriction of $\tau$ to $\mr M=M\setminus U$ (recall that $U$ is a small tubular neighborhood of the collection $c$ of singular orbits of $\phi$). That way we will have room to work in the solid tori of $U$. Homologically this changes nothing, and we will use the notation $\mr\tau$ and $B_{\mr\tau}$ to denote $\tau\setminus U$ and $B_\tau\setminus U$ respectively. A taut ideal triangulation of $\mr M$ will mean a taut ideal triangulation of $M\setminus c$ restricted to $\mr M$.

\subsection{On the boundary of $ \mathring M $}

\begin{figure}
\centering
\includegraphics[height=3in]{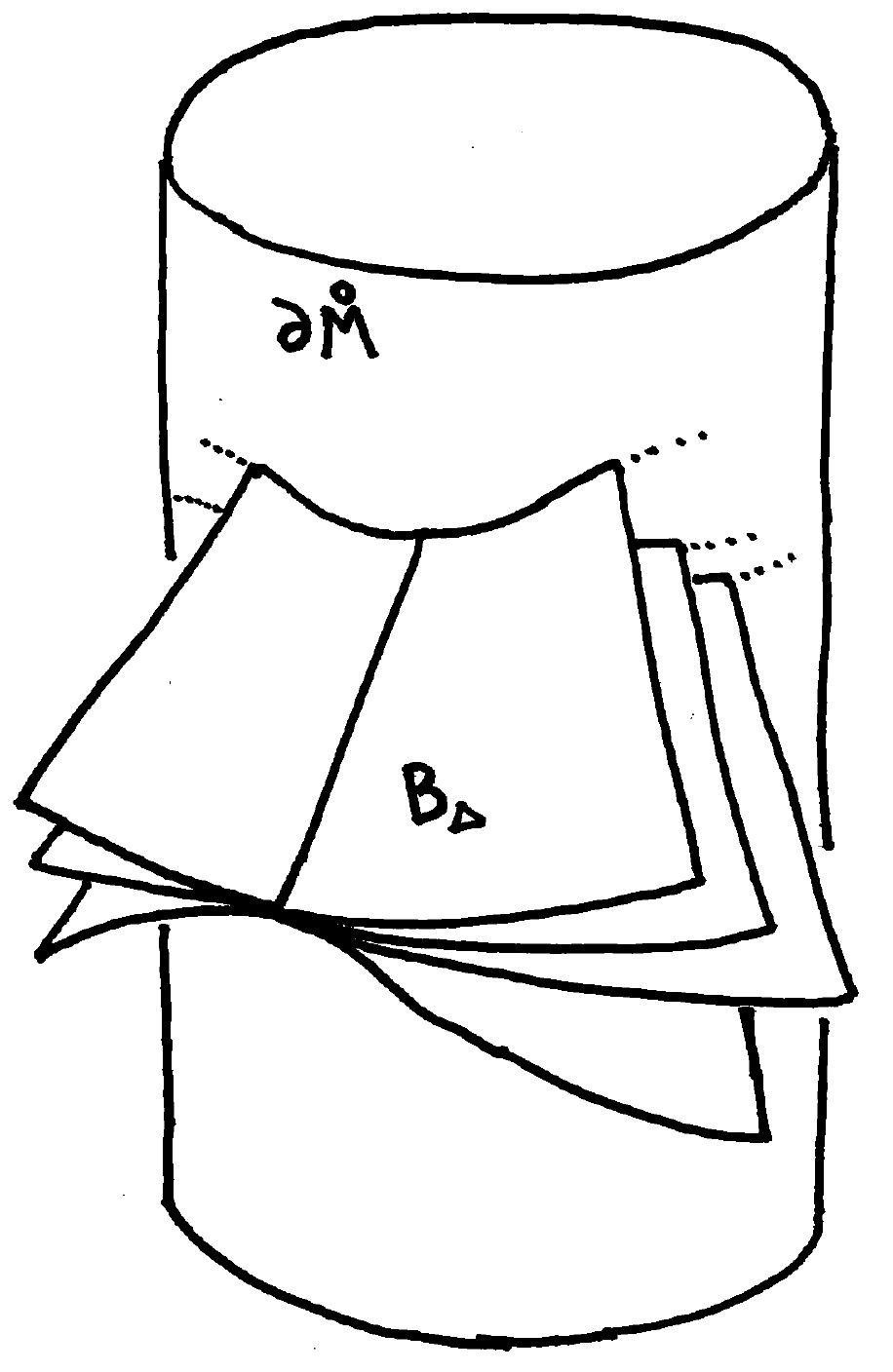}
\caption{}
\label{lookingatT}
\end{figure}

Let $\triangle $ be a taut ideal triangulation of $\mathring{ M} $. Then we may assume its 2-skeleton $ B_\triangle $ intersects $\partial  \mathring M $ transversely in a co\"oriented train track as in Figure \ref{lookingatT}. This train track divides each component of $\partial \mr M$ into bigons, which we will think of as triangles with one vertex whose interior angle is $\pi$  (corresponding to a smooth path through a switch), and two cuspidal vertices whose angles are 0.

A triangle as above, with two vertices of interior angle 0 and one with angle $ \pi$, is called \textbf{flat}. The edge between the two 0-vertices is called a \textbf{0-0 edge} and the other two are called \textbf{0-$\boldsymbol\pi$ edges}. We say the triangle is \textbf{upward} or \textbf{downward} if the train track's co\"orientation points out of or into a flat triangle at its $\pi  $-vertex,  respectively. See Figure \ref{upward downward}.

\begin{figure}
\centering

\begin{tikzpicture}[scale=1.4]
\draw [line width=2pt](0,0) .. controls (.5,-.05) and (.5,.5) .. (1,.5)
.. controls (1.5,.5) and (1.5,-.05) .. (2,0);
\draw [line width=2pt](0,0) .. controls (.5,-.1) and (1.5,-.1) .. (2,0);
\node at (0,0)[circle, fill=black][scale=0.4]{};
\node at (1,.5)[circle, fill=black][scale=0.4]{};
\node at (2,0)[circle, fill=black][scale=0.4]{};

\draw [line width=2pt](3,.5) .. controls (3.5,.55) and (3.5,0) .. (4,0)
.. controls (4.5,0) and (4.5,.55) .. (5,0.5);
\draw [line width=2pt](3,.5) .. controls (3.5, .6) and (4.5, .6) .. (5,.5);
\node at (3,0.5)[circle, fill=black][scale=0.4]{};
\node at (4,0)[circle, fill=black][scale=0.4]{};
\node at (5,.5)[circle, fill=black][scale=0.4]{};

\end{tikzpicture}

\caption{On the left and right, respectively, are upward and downward flat triangles. The co\"orientation of all edges points upward.}
\label{upward downward}
\end{figure}
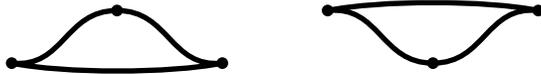

We draw the train track with the convention that we view $\partial \mr M$ from inside $\mr M$ and the co\"orientation points upward. Having made this convention we define the \textbf{triangles to the left} and \textbf{to the right} of a switch $p$ to be the triangles with interior angle 0 at $ p $ which lie to the left and right of $ p $.

\begin{lemma}[Veering condition on boundary]\label{veering on boundary}
Let $ \triangle$ be a taut ideal triangulation of $\mathring M$. Then $\triangle  $ is veering if and only if the train track on $\partial  \mathring M $ has the property that for every switch $ p $ the triangles to the left of $ p $ are all upward or all downward, and the triangles to the right of $ p $ are all the opposite.
\end{lemma}

\begin{proof}
Each switch $p$ in the train track corresponds to an edge $e$ in the triangulation, and it is clear that with respect to a single component of $\partial \mathring M$, the intersections with tetrahedra to the right of the switch are flat triangles of the same type if and only if the corresponding tetrahedra are of the same type when $e$ is distinguished. The triangles to the left of $p$ will be of the opposite type if and only if their corresponding tetrahedra are of the same type as those to the right of $p$.
\end{proof}

\begin{figure}
\centering
\includegraphics[height=1.8in]{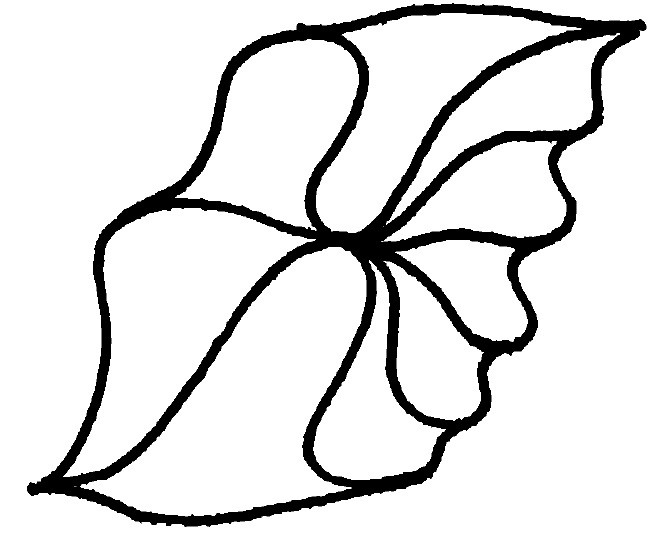}
\caption{The edge of the taut ideal triangulation corresponding to the central switch satisfies the veering condition.}
\label{veeringonboundary}
\end{figure}

An example of a switch satisfying the condition of Lemma \ref{veering on boundary} is shown in Figure \ref{veeringonboundary}.
In \cite{FG}, Futer and Gu\'eritaud observed the following structure in the intersection of a veering triangulation with $\partial \mathring M$.

\begin{lemma}\label{bands}
Fix a component $ T_i $ of $\partial \mr M$. Let $ u , d \subset T_i$ be the closed regions consisting of upward and downward triangles, respectively. Then $u$ and $d$ are collections of essential annuli.
If $t$ is a triangle, then each of $t$'s 3 vertices lie in $\partial  u =\partial  d$. 
\end{lemma}

For example, if $ t $ is upward then each edge of $ t $ either lies entirely in $\partial  u =\partial  d $ or traverses $ u $ with its endpoints on the boundary.

\begin{proof}
The reader can check that a violation of the lemma contradicts Lemma \ref{veering on boundary}.
\end{proof}

We will call the annulus connected components of $ u $ and $ d $ \textbf{upward} and \textbf{downward bands}, respectively. Part of the content of Lemma \ref{bands} is that each band is only one edge across. Gu\'eritaud calls the edges forming the boundaries of bands \textbf{ladderpole edges}, and the edges which traverse the bands \textbf{rungs}. For an example of one of these triangulations lifted to $\mathbb R^2$, see Figure \ref{lifted train track}.

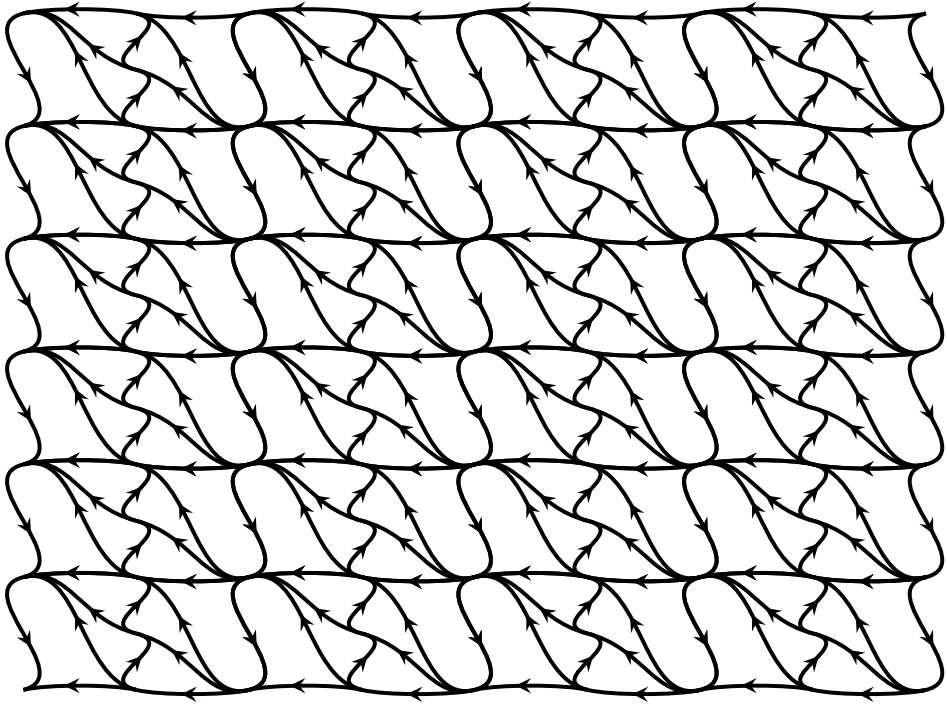
\begin{figure}
\centering

\begin{tikzpicture}[scale=.75]
\foreach \x in {0,4,8,12}
\foreach \y in {0,2,...,10}
{
\draw [-<-=.5, line width=1.5pt] (\x+0,\y+0) .. controls (\x+1,\y+.25) and (\x-1,\y+1.75) .. (\x+0,\y+2);
\draw [-<-=.5, line width=1.5pt](\x+0,\y+2) .. controls (\x+1,\y+2.25) and (\x+1,\y+.2) .. (\x+2,\y+0);
\draw [-<-=.7, line width=1.5pt](\x+0,\y+2) .. controls (\x+.8,\y+2.2) and (\x+1,\y+1.25) .. (\x+2,\y+1);
\draw [->-=.6, line width=1.5pt](\x+2,\y+0) .. controls (\x+1.2,\y+.2) and (\x+2.8,\y+.8) .. (\x+2,\y+1);
\draw [->-=.6, line width=1.5pt](\x+2,\y+1) .. controls (\x+1.2,\y+1.2) and (\x+2.8,\y+1.8) .. (\x+2,\y+2);
\draw [-<-=.5, line width=1.5pt](\x+0,\y+0) .. controls (\x+.4,\y+.1) and (\x+1.6,\y+.1) .. (\x+2,\y+0);
\draw [-<-=.5, line width=1.5pt](\x+0,\y+2) .. controls (\x+.4,\y+2.1) and (\x+1.6,\y+2.1) .. (\x+2,\y+2);
}
\foreach \x in {4,8,12,16}
\foreach \y in {2,4,...,12}
{
\draw [->-=.61,line width=1.5pt] (\x+0,\y+0) .. controls (\x-1,\y-.25) and (\x+1,\y-1.75) .. (\x+0,\y-2);
\draw [->-=.65, line width=1.5pt] (\x+0,\y-2) .. controls (\x-1,\y-2.25) and (\x-1,\y-.2) .. (\x-2,\y+0);
\draw [->-=.7, line width=1.5pt] (\x+0,\y-2) .. controls (\x-.8,\y-2.2) and (\x-1,\y-1.25) .. (\x-2,\y-1);
\draw [->-=.6, line width=1.5pt] (\x+0,\y+0) .. controls (\x-.4,\y-.1) and (\x-1.6,\y-.1) .. (\x-2,\y+0);
\draw [->-=.6, line width=1.5pt] (\x+0,\y-2) .. controls (\x-.4,\y-2.1) and (\x-1.6,\y-2.1) .. (\x-2,\y-2);
}
\node at (17,0)[circle,fill=white]{};
\end{tikzpicture}

\caption{An example of a train track coming from a veering triangulation, endowed with its inherited orientation, lifted to the plane. This particular example comes from the veering triangulation associated to a minimal dilation 4-strand braid, discussed in \cite{Ag}. This train track is associated to the cusp coming from the suspension of the monodromy's lone singular point.
}
\label{lifted train track}
\end{figure}

\section{Moving forward with Oertel's question}

\subsection{More on the structure of $\partial \mr\tau$}\label{boundary train track stuff}

Define $\gamma:=\partial \mr\tau$, i.e. the intersection of the veering triangulation with $\partial \mr M$. Fix a component $T_i$ of $\partial \mr M$, and set $\gamma_i=\gamma\cap T_i$. We say the slope $s_i$ corresponding to the union of all ladderpole edges in $\gamma_i$ is the \textbf{ladderpole slope} for $T_i$. It will be convenient to think of $\gamma_i$ as oriented in addition to being co\"oriented, meaning that each edge of $\gamma_i$ has a preferred direction, and the preferred directions are compatible at switches. Our choice of orientation is the one induced by the boundary of any surface carried by $B_{\mr \tau}$. In particular, the orientation on all rungs of $\gamma_i$ is from right to left (as indicated in Figure \ref{lifted train track}).
We will say $\gamma_i$ \textbf{positively carries} an oriented curve $a$ if the orientation of $a$ agrees with the orientations of the edges in $\gamma_i$.
Define an \textbf{upward} (resp. \textbf{downward})\textbf{ladderpole} to be the right (resp. left) boundary component of an upward band in $T_i$, endowed with the orientation it inherits from $\gamma_i$. In our pictures, the upward ladderpoles are oriented upwards.

\begin{lemma}\label{traintrackfacts}
Let $s_i$ and $T_i$ be as above. Then
\begin{enumerate}
\item $\gamma_i$ positively carries curves $s_i^+$ and $s_i^-$ with slope $s_i$ such that $[s_i^+]+[s_i^-]=0$ in $H_1(T_i)$, and $s_i$ is the unique such slope on $T_i$; 

\item any curve carried by $\gamma_i$ with slope $s_i$ traverses only ladderpole edges; and

\item $\gamma_i$ positively carries a representative of every integral class in a half space of $H_1(T_i)$ bounded by $\mathbb R\cdot [s_i^+]$.

\end{enumerate}
\end{lemma}

\begin{proof}
The union of all edges in a positive ladderpole forms an oriented curve $s_i^+$ which is positively carried by $\gamma_i$, and similarly a downward ladderpole gives an oriented positively carried curve $s_i^-$. It is clear that $[s_i^+]$ and $[s_i^-]$ sum to 0 in homology. 

Let $\alpha$ be a curve positively carried by $\gamma_i$ which traverses a rung of $\gamma_i$, and recall that all rungs of $\gamma_i$ are oriented from right to left. As we trace along $\alpha$, we must traverse a rung of every band in $T_i$ from right to left, as otherwise $\alpha$ could not close up. Therefore $\alpha$ intersects $s_i^+$, and after a perturbation we can assume that the intersections are all positive, so $i([s_i^+],[\alpha])>0$. Hence any curve carried positively or negatively by $\gamma_i$ which traverses a rung has nonzero  intersection number with $s_i^+$ and cannot have slope $s_i$. This proves claim 2. 

If $\beta$ is a curve carried positively by $\gamma_i$ with slope $\ne s_i$, then it traverses rungs of $\gamma_i$ and has positive intersection with $s_i^+$ as above. This means $\gamma_i$ cannot positively carry a representative of $-[\beta]$, completing the proof of claim 1.

Convex combinations of classes with positively carried representatives can be represented by oriented cut and paste sums carried by $\gamma_i$. Also, if $\gamma_i$ positively carries a curve representing $ka$ for $a\in H_1(T_i)$, $k\in\mathbb Z_{>1}$, then $\gamma_i$ positively carries a curve representing $a$. Indeed, if $\rho$ is such a curve, we can perform oriented surgeries to eliminate self-intersections that do not change the homology class of $\rho$, yielding $k$ parallel curves carried by $\gamma_i$ which represent $a$. In other words, to show that a homology class is represented by a positively carried curve, it it enough to show some multiple of the homology class is represented by a positively carried curve.

Therefore to prove claim 3, it suffices by claims 1 and 2 to show that $\gamma_i$ positively carries a curve traversing a rung of $\gamma_i$. It is easy to see that there is such a curve: you can draw one by starting at a switch and tracing along rungs of $\gamma_i$ until you reach the ladderpole containing your path's initial point. Then the path can be closed up along the ladderpole.
\end{proof}

Lemma \ref{traintrackfacts} concerned all curves carried by $\gamma$. Now we consider the boundaries of surfaces carried by $B_{\mr \tau}$.
If an oriented surface is carried by $B_{\mr\tau}$, its boundary traces out an oriented collection of curves on $\partial \mathring M$ that is positively carried by $\gamma$.
Our understanding of $\gamma$ allows us to deduce information about the surfaces carried by $B_{\mr\tau}$, and thus surfaces representing classes in $\sigma$.

For example, we observe that every surface carried by $B_{\mr \tau}$ has a non-ladderpole boundary component.

\begin{lemma}\label{not all exceptional}
There is no surface carried by $B_{\mr\tau}$ whose boundary components are all ladderpoles. 
\end{lemma}

\begin{proof}
Suppose $S$ is a surface carried by $B_{\mr\tau}$. An edge of $\gamma$ traversed by $\partial S$ corresponds to a truncated ideal triangle of ${\mr\tau}$ on which $S$ has positive weight. Let $t$ be such a triangle. Choose a truncated ideal tetrahedron $Y$ of which $t$ is a face. The truncated vertices of $t$ correspond to 3 edges of the truncated vertices of $Y$, which are flat triangles lying in $\partial \mr M$. Exactly one of these edges, which we will call $e$, is a 0-0 edge with respect to a truncated vertex of $Y$.

Each edge in $\gamma$ is incident to two flat triangles. Our knowledge of $\gamma$ gives us that a ladderpole edge is 0-$\pi$ with respect to both of its incident triangles, while a rung is 0-0 with respect to one incident triangle and 0-$\pi$ with respect to the other. It follows that $e$ is a rung, so $\partial S$ cannot consist of all ladderpole edges. 
\end{proof}

Recall from Section \ref{notation} the injective puncturing map $P\colon H_2(M)\rightarrow H_2(\mathring M, \partial \mr M)$, induced at the level of chains by sending $S$ to $S\setminus U$. As a reminder, we will write $P(\alpha)=:\mathring\alpha$ when convenient. 

We can classify the possible boundaries of all surfaces carried by $B_{\mr\tau}$ coming from $M$, i.e. those representing classes that lie in the image of $P$.

\begin{lemma} \label{types of boundaries}
Let $\alpha\in \cone(\sigma)$ be an integral class, and consider a surface $Y\subset \mr M$ carried by $B_{\mr\tau}$ and representing $\mr\alpha$ . Then for each component $T_i$ of $\partial \mr M$, $\partial Y\cap T_i$ is positively carried by $\gamma_i$, and is either:
\begin{enumerate}
\item empty,
\item a collection of ladderpoles which is nulhomologous in $T_i$, or
\item a collection of meridians.
\end{enumerate}
Further, if $\alpha\in \intr \cone(\sigma)$, then $\partial Y\cap T_i$ is a nonempty collection of meridians.
\end{lemma}

\begin{proof}
Because the co\"orientation of $Y$ agrees with that of $B_{\mr\tau}$, the orientation of $\partial Y\cap T_i$ agrees with that of $\gamma_i$. Hence $\partial Y\cap T_i$ is positively carried.

Let $Q$ be the map $H_2(M)\to H_2(\closure (U),\partial \mr M)$ defined similarly to the puncturing map $P$ by excising $M\setminus \closure (U)$. Then we have the commutative diagram
\begin{equation}\label{commutative diagram}
\begin{tikzcd}
H_2(M)\arrow[hookrightarrow]{r}{P}\arrow{d}{Q}&H_2(\mr M,\partial \mr M)\arrow{d}{\partial}\\
H_2(\closure(U),\partial \mr M)\arrow{r}{-\partial} &H_1(\partial \mr M)
\end{tikzcd}
\end{equation}
(A quick remark: the minus sign on the bottom boundary map of (\ref{commutative diagram}) reflects the fact that $\partial \colon H_2(\closure(U),\partial \mr M)\rightarrow H_1(\partial\mr M)$ is induced by the boundary map on 2-chains \emph{inside} $U$, while the vertical boundary map is induced by the map on 2-chains \emph{outside} $U$.)

Fix a component $U_i$ of $U$ with boundary $T_i$, meridional disk $D_i$ and set $m_i=-\partial D_i$.  Since the images of $\partial\circ P$ and $-\partial \circ Q$ are equal, and $H_2(\closure(U_i),T_i)=\langle [D_i]\rangle$, the boundary of $Y$ on $T_i$ is homologous to $km_i$, where $k\in \mathbb Z$. Moreover, $\partial Y$ is embedded and carried by $\gamma$. In particular this means the curves of $\partial Y\cap T_i$ are parallel.

If $\partial Y\cap T_i$ is nonempty and nulhomologous in $T_i$, then it must consist of an even number of ladderpoles, by Lemma \ref{traintrackfacts}.

Otherwise $k\ne 0$, in which case $\partial Y\cap T_i$ must be a collection of $k$ meridians. If $\alpha\in \intr\cone(\sigma)$, then by Fried's results in \cite{Fr}, $\alpha$ has positive intersection with each singular orbit of $\phi$. Thus $k> 0$ in the above analysis, completing the proof.
\end{proof}

A \textbf{vertex class} of $\sigma$ is a primitive integral homology class $v\in H_2(M)$ projecting to a vertex of $\sigma$.

Let $T_i$ be a torus boundary component of $\mathring{M}$. An embedded surface $Y\subset \mathring{M}$ carried by $B_{\mr\tau}$ is \textbf{ladderpole at $T_i$} if $Y\cap T_i$ is a collection of ladderpoles. The homology class $\alpha\in H_2(M)$ is \textbf{ladderpole at $T_i$} if $\mr\alpha$ has an embedded representative carried by $B_{\mr\tau}$ which is ladderpole at $T_i$.

Note that by Lemma \ref{types of boundaries}, ladderpole classes must lie in $\partial \cone(\sigma)$.

\subsection{Our approach to Question \ref{Oertel's question}}

In \cite{Mo}, Mosher gives an example of how to construct a branched surface spanning any fibered face $F$ of a 3-manifold $N$. He simply takes embedded minimal-$\chi_-$ representatives $S_i$ of each vertex class $v_i$ for the face and perturbs them to intersect transversely. Then there is an isotopy in a neighborhood of the intersection locus, shown in Figure \ref{branchedsum2}, which gives the union of the surfaces the structure of a branched surface $B_F$ carrying each vertex class. We will refer to this operation on transverse co\"oriented surfaces as \textbf{branched sum}. Since the vertex classes span $F$, the $B_F$ spans $F$. The problem with this method, as Mosher notes, is that there is no guarantee that all surfaces carried by the $B_F$ are incompressible.

\begin{figure}
\centering
\includegraphics[height=1.9in]{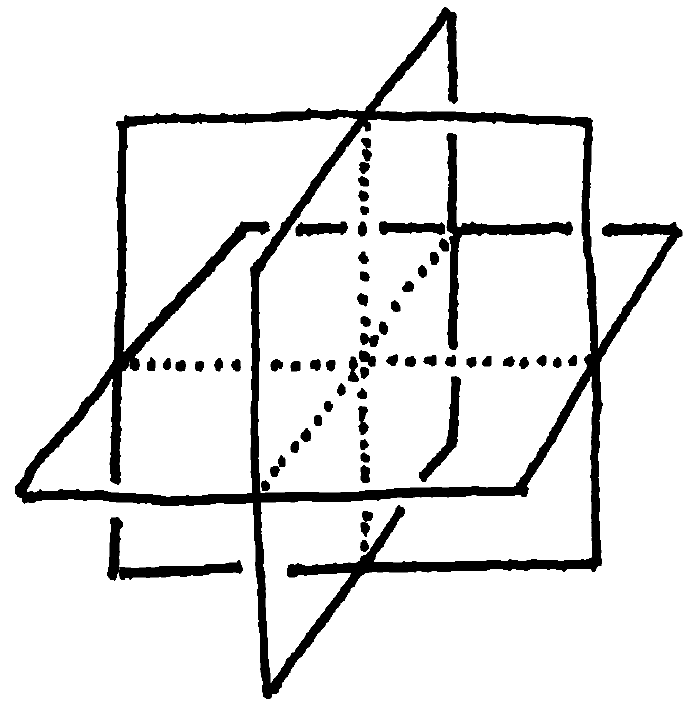}
\quad\quad
\includegraphics[height=1.9in]{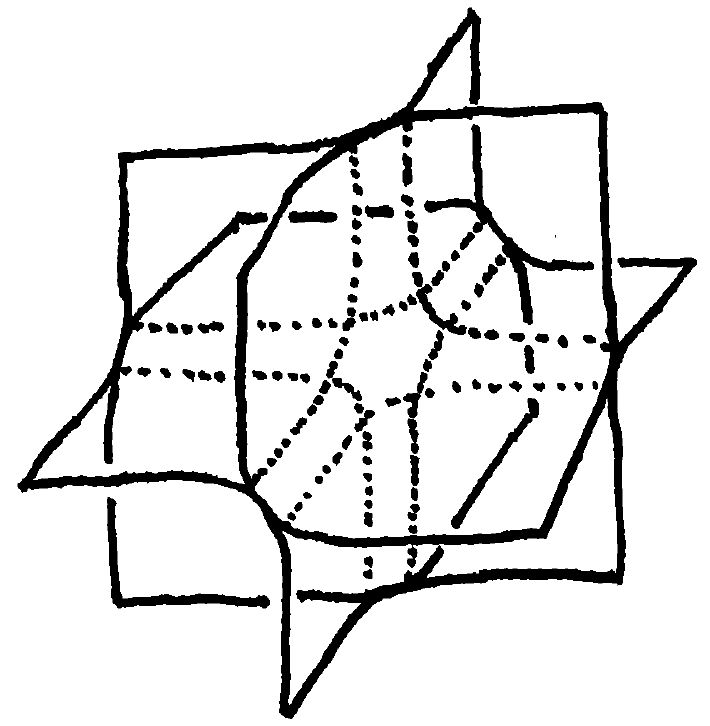}
\caption{The branched sum operation. A generic intersection, shown on the left, can be smoothed in a unique way, shown on the right, such that it preserves co\"orientations. Here the co\"orientations are such that they all point into the octant facing the reader.}
\label{branchedsum2}
\end{figure}

However, $B_F$ is frequently almost taut. Indeed, consider the surface $\sum_iS_i$, i.e. the oriented cut-and-paste sum of the $S_i$. We have 
\[
\chi_-(\sum_i S_i) =\sum_i \chi_-(S_i)=\sum_i x(v_i)=x(\sum_i v_i),
\]
where the last equality follows from the fact that $x$ is linear on $\cone (F)$. Therefore $\sum_i S_i$ realizes the minimal $\chi_-$ in the homology class $\sum_i v_i$. Moreover, $\sum_i S_i$ is fully carried by $B_F$, so almost tautness follows from Lemma \ref{tautness from carrying} as long as $B_F$ carries no disks or spheres. Slightly more generally, we have proved the following lemma.

\begin{lemma}\label{branched sum almost taut}
Let $F$ be a closed face of the Thurston norm ball in $H_2(N,\partial N)$. If $\{S_i\}$ is a finite collection of norm-minimizing surfaces embedded in $M$ such that $\{[S_i]\}\subset \cone (F)$, then the branched sum of the $S_i$ is almost taut provided it carries no disks or spheres.
\end{lemma}

If $B_F$ is indeed almost taut and carries no disks or spheres, it follows that the only way $B_F$ can fail to be an answer to Oertel's question \ref{Oertel's question} is if it carries a compressible torus. 

The method we use to address Oertel's Question \ref{Oertel's question} is similar to Mosher's example above. The difference is that rather than take any embedded representatives of vertex classes, we take representatives of their punctured images in $H_2(\mathring{M},\partial \mathring{M})$ lying in a regular neighborhood of $\tau$. This is possible because by Theorem \ref{Agol's theorem}, $B_{\mr\tau}$ carries a representative of $\mr\alpha$ for every integral class in $\intr \cone(\sigma)$. By Lemma \ref{closed cone}, the same can be said for representatives of classes $\mr\beta$ for $\beta\in \partial \cone(\sigma)$. 
With an extra hypothesis, we show how to extend the surfaces over the Dehn filling so that their branched sum $B_\sigma$ carries no tori. Rather than show directly that no tori are carried, our method of proof is to demonstrate that $B_\sigma$ is a homology branched surface, which is enough to imply tautness by Lemmas \ref{tautness from carrying} and \ref{almost taut to taut}.

\subsection{A lemma concerning $x$ and Dehn filling}

We are particularly interested in the restriction of the puncturing map $P\colon H_2(M)\to H_2(\mr M,\partial \mr M)$ to $\cone(\sigma)$, and it will be useful to have the following lemma concerning $P$'s relationship to the norm $x$. The subject of how the Thurston norm behaves under Dehn filling has been studied in \cite{Gab87a}, \cite{Gab87b}, and \cite{Sel}, and more recently in \cite {BakTay}.

Recall that $c$ denotes the union of the singular orbits of $\phi$, which is the flow associated to $\sigma$.

\begin{lemma}\label{norm and filling}
Let $\alpha\in \cone(\sigma)$ be an integral homology class. Then 
\[
x(\alpha)=x(\mathring \alpha)-i(\alpha,c), 
\]
where $i(\,\, \,, \,\, )$ is algebraic intersection number.
\end{lemma}

Before proving Lemma \ref{norm and filling} we state some results from \cite{Mo} that require some definitions.

We say a flow $\phi'$ is a \textbf{dynamic blowup} of $\phi$ if is obtained by the following procedure: we replace a singular orbit $\theta$ by the suspension of a homeomorphism $f$ of a finite tree $T$. This homeomorphism's first return map on each edge of $T$ should fix the endpoints and act without fixed points on the interior. Thus each edge of $T$ can be given an orientation according to the direction points are moved by its first return map, and around each vertex these orientations should alternate between outward and inward. The suspended tree forms a complex of annuli $K$ which is invariant under $\phi'$. The new flow $\phi'$ is semiconjugate to $\phi$ by a map which collapses $K$ to $\theta$, and is one to one in the complement of $K$. The vector fields generating $\phi$ and $\phi'$ differ only inside a small neighborhood of $\theta$.

We say that a surface $S$ embedded in $M$ is \textbf{almost transverse} to $\phi$ if there is a dynamic blowup $\phi^\#$ of $\phi$ such that $S$ is transverse to $\phi^\#$, and the sum of tangent spaces $TS\oplus T\phi^\#$ is positively oriented in $M$. Here $TS$ and $T\phi^\#$ denote the tangent spaces to the surface $S$ and flow $\phi^\#$, respectively.

\begin{lemma}[\cite{Mo}]\label{almosttransverse1}
Let $\alpha\in H_2(M)$ be an integral class. Then $\alpha$ can be represented by a surface almost transverse to $\phi$ if and only if $\alpha\in\cone \sigma$ (recall that $\phi$ depends on $\sigma$). More specifically, there exists a way to dynamically blow up $\phi$ to a flow $\phi^\#$ along only singular orbits $\theta$ of $\phi$ with $i(\alpha, \theta)=0$ such that $\alpha$ is represented by a surface $S$ transverse to $\phi^\#$.
\end{lemma}

Singular orbits $\theta$ with $i(\alpha,\theta)=0$ are called \textbf{$\alpha$-null}, and if $i(\alpha,\theta)>0$ we will say that $\theta$ is \textbf{$\alpha$-positive}. It will also be useful to know:

\begin{lemma}[\cite{Mo}]\label{almosttransverse2}
Let $S$ be a surface almost transverse to a pseudo-Anosov flow on $M$. Then $S$ is norm-minimizing.
\end{lemma}

\begin{proof}[Proof of Lemma \ref{norm and filling}]
First we will show that $x(\alpha)\le x(\mathring\alpha)-i(\alpha,c)$.

Let $Y$ be a representative of $\mathring \alpha$ carried by $B_{\mr\tau}$, and let $T_i=\partial U_i$ be one of the torus boundary components of $\mr M$. By Lemma \ref{types of boundaries}, $\partial Y\cap T_i$ is either a collection of meridians or a collection of ladderpoles that represents 0 in $H_2(T_i)$. Such a nulhomologous collection can be realized as the boundary of a family of closed annuli whose interiors are embedded in $U_i$.

Observe that $x(\alpha)\le x(\mathring\alpha)-n(Y, \partial\mr M)$, where $n$ is the number of meridians of $\partial\mr M$ in $\partial Y$. Indeed, we can glue a disk to each meridian boundary component, and cap off all other boundary components with annuli. Because $B_{\mr\tau}$ is taut and thus $Y$ is norm-minimizing, we obtain a representative $S$ of $\alpha$ with $\chi_-(S)=x(\mr\alpha)-n(Y,\partial\mr M)$.

Now we claim that $n(Y,\partial\mr M)=i(\alpha,c)$.

First, note that by the proof of Lemma \ref{types of boundaries}, the images of the boundary maps $\partial |_{\text{image(P)}} :H_2(\mr M,\partial\mr M) \rightarrow H_1(\partial\mr M)$ and $\partial : H_2(\closure(U),\partial\mr M)\rightarrow H_1(\partial\mr M)$ are contained in the subgroup $\bigoplus_i \langle m_i \rangle$ of $H_1(\partial\mr M)$. Here $m_i$ is the homology class of a meridional curve on $T_i$, oriented so that it is positively carried by $\gamma_i$. Let $\pi_i: \langle m_i\rangle\to \Z$ be the map $km_i\mapsto k$. Because $Y$ is carried by $B_{\mr\tau}$, then $n(Y,\partial\mr M)=\sum_i \pi_i(\partial(\mr\alpha))$, and it thus makes sense to write $n(\mr\alpha, \partial\mr M)$. 

We can package this information into an updated version of the diagram (\ref{commutative diagram}):

\begin{center}
\begin{tikzcd}
H_2(M)\arrow{rr}{P}\arrow{d}{Q}&&H_2(\mr M,\partial\mr M)\arrow{d}{\partial}\\
H_2(\closure(U),\partial\mr M)\arrow{rr}{-\partial} \arrow[densely dotted]{rd}[swap]{i(-,c)}&&\bigoplus_i \langle m_i \rangle \arrow[densely dotted]{dl}{\sum\pi_i}\\
&\Z&
\end{tikzcd},
\end{center}
where the dotted arrows are defined on the integral lattice of their domains. 
Since $i(\alpha,c)=i(Q(\alpha),c)$, the claim that $n(\mr \alpha,\partial\mr M)=i(\alpha,c)$ reduces to the claim that the above diagram is commutative, which is true since the square and triangle commute. Therefore $x(\alpha)\le x(\mathring\alpha)-i(\alpha,c)$.

Next we will show  that $x(\mr\alpha)\le x(\alpha)+i(\alpha,c)$. The idea of this direction is to take a norm-minimizing surface transverse to $c$ representing $\alpha$ and delete its intersections with $U$. This will give a representative of $\mr\alpha$ with $\chi_-=x(\alpha)+i(\alpha,c)$ as long as the orientations of all intersections with $c$ agree. Hence it suffices to show that there exists such a norm-minimizing surface representing $\alpha$, and this is where we use Lemmas \ref{almosttransverse1} and \ref{almosttransverse2}.

Let $S$ be a representative of $\alpha$ which is almost transverse to $\phi$ with blown up flow $\phi^\#$, let $\theta$ be an $\alpha$-null orbit of $\phi$, let $K_\theta$ be the $\phi^\#$-invariant annulus complex blowing up $\theta$, and $U_\theta^\#$ a solid torus containing $K_\theta$. Since $S$ is transverse to $\phi^\#$ and $TS\oplus T\phi^\#$ is positively oriented, we have in particular that each intersection point of $S$ with any $\alpha$-positive singular orbit of $\phi^\#$  is positively oriented. This is close to the property we want, but we need it to hold for $\phi$ and not just $\phi^\#$. Thus we will show that $S$ intersects the solid torus $U_\theta^\#$ in a collection of disjoint annuli, and can be isotoped outside $U_\theta^\#$. Then $S$'s positivity with respect to the $\alpha$-positive singular orbits of $\phi$ will be preserved by the semiconjugacy collapsing $K_\theta$ to $\theta$. Arguing in this way for each $\alpha$-null orbit, we will obtain a norm-minimizing surface representing $\alpha$ that has only positively oriented intersection points with $c$.

Observe that $[S]$ maps to $0$ under the map $H_2(M)\to H_2(\closure(U_\theta^\#),\partial U_\theta^\#)$ because $\theta$ is $\alpha$-null. After applying the boundary map $H_2(\closure(U_\theta^\#),\partial U_\theta^\#)\to H_1(\partial U_\theta^\#)$ we see that $S\cap\partial U_\theta^\#$ represents $0$ in $H_1(\partial U_\theta^\#)$. It follows that $S\cap \partial U_\theta^\#$, after perturbation for transversality, is a possibly empty nulhomologous embedded collection of curves. If it is empty, there is nothing for us to prove. 

If $S\cap \partial U_\theta^\#$ contains a curve $a$ which is inessential in $\partial U_\theta^\#$, then $a$ bounds a disk in $\partial U_\theta^\#$. If $a$ is essential in $S$,  this contradicts the incompressibility of $S$. If $a$ is inessential in $S$, then we can perform an isotopy of $S$ which removes $a$ from $\partial U_\theta^\#$. Therefore we can assume $S\cap \partial U_\theta^\#$ is a collection of curves essential in $\partial U_\theta^\#$. Since the collection is embedded  and nulhomologous, there must be $2n$ curves of the same slope, half having one orientation and half the other.

If this slope is the meridional slope for $U_\theta^\#$, the norm-minimality of $S$ implies that $S\cap \closure( U_\theta^\#)$ is a collection of $2n$ disks. However, since the orientations of these disks must match up with those of $S\cap \partial U_\theta^\#$, $n$ of the disks must intersect flow lines of $\phi^\#$ negatively, contradicting the definition of almost transversality.

The last possibility is that these curves do not have meridional slope, i.e. do not bound disks in $U_\theta^\#$. The norm-minimality of $S$ then implies that $S\cap \closure(U_\theta^\#)$ is a collection of annuli. As we noted before, this means $S_\alpha$ can be isotoped outside of $U_\theta^\#$, completing the proof.
\end{proof}

\subsection{Main theorem}

\begin{theorem}
\label{THBS from cusp condition}
If every boundary torus of $\mathring{M}$ has at most two ladderpole vertex classes, then $M$ has a taut homology branched surface spanning $\sigma$.
\end{theorem}

\begin{proof}
Let $v_1,\dots, v_n$ be the vertex classes of $\sigma$. Take embedded representatives $\mathring{S}_1,\dots, \mathring{S}_n$ of $\mathring{v}_1,\dots, \mathring{v_n}$ which lie in a regular neighborhood of $B_{\mr\tau}$, transverse to its vertical foliation. 
After performing the Dehn filling at each boundary torus, these surfaces with boundary are embedded in $M$. By capping off their boundaries as follows, we can extend them over the Dehn filling so that they represent $v_1,\dots,v_n$.

For each boundary torus $T_i$ which $\mr S_j$ meets in a collection of meridians, we glue an embedded family of disks $D_{i,j}\subset \closure (U_i)$ to $\mr S_j\cap T_i$. 

If $\mathring{S_j}$ is ladderpole at $T_i$ for some $i$, then $\mathring{S_j}\cap T_i$ is an even-sized  collection of co\"oriented ladderpoles which sums to 0 in the first homology of $T_i$. Thus we may glue in a disjoint collection of annuli embedded in $U_i$ whose co\"orientation matches that of $\mathring{S_j}\cap T_i$. Call this collection $A_{i,j}$. 

By assumption, at most one other vertex class is ladderpole at $T_i$; if $\mathring{S_k}$ is also ladderpole at $T_i$, we take another such collection of annuli, $A_{i,k}$. We isotope the two families rel $T_i$ so that they intersect essentially. These families of annuli, $A_{i,j}$ and $A_{i,k}$, have ladderpole boundaries by the construction, and in particular they cut the solid torus $U_i$ into smaller solid subtori. By possibly choosing a new $A_{i,j}$ and $A_{i,k}$, we assume that they cut $U_i$ into a minimal number of solid subtori among all possible choices.

The boundaries of these solid subtori are partitioned into subannuli which are subsets of either $A_{i,j}$, $A_{i,k}$, or $\partial U_i$; the subannuli coming from $A_{i,j}$ and $A_{i,k}$ are co\"oriented. 

We now describe configurations of annuli which we call \emph{sources} and \emph{sinks}, and we will then explain why the $A_{i,j}$ and $A_{i,k}$ do not form these configurations. A source is a solid subtorus of $U_i$ having a boundary composed of subannuli from $A_{i,j}$ and $A_{i,k}$ whose co\"orientations point out of the subtorus. Similarly, a sink has co\"oriented boundary subannuli all pointing inward. Note that because $A_{i,j}$ is embedded, no two subannuli from $A_{i,j}$ are adjacent on the boundary of a source or sink; the same holds for $A_{i,k}$. Given a source or sink, we can perform an oriented cut and paste surgery on $A_{i,j}$ and $A_{i,k}$ as shown in Figure \ref{source sink}. As explained in the caption of Figure \ref{source sink}, this surgery reduces the number of solid subtori of $U_i$. Because we chose $A_{i,j}$ and $A_{i,k}$ to minimize the number of subtori, we conclude there are no sources or sinks in $U_i$.

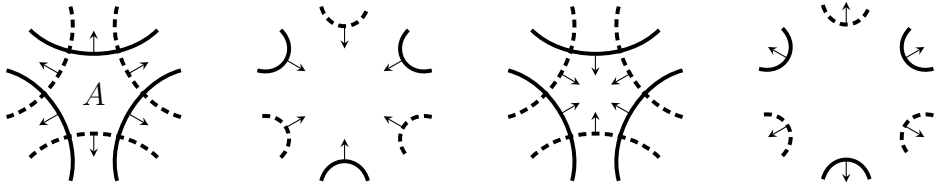
\begin{figure}
\centering
\begin{tikzpicture}[scale=1.2]
\draw [line width=1.5pt](45:1) .. controls (45:.5) and (135:.5) .. (135:1);
\draw [line width=1.5pt](165:1) .. controls (165:.5) and (255:.5) .. (255:1);
\draw [line width=1.5pt](285:1) .. controls (285:.5) and (15:.5) .. (15:1);
\draw[line width=1.5pt,densely dashed] (345:1) .. controls (345:.5) and (75:.5) .. (75:1);
\draw [line width=1.5pt,densely dashed] (105:1) .. controls (105:.5) and (195:.5) .. (195:1);
\draw [line width=1.5pt,densely dashed] (225:1) .. controls (225:.5) and (315:.5) .. (315:1);
\foreach \x in {30,90,...,330}
	\draw[->] (\x:.45)--(\x:.7);
\node at (0,0) {$A$};
\end{tikzpicture}
\quad\quad
\begin{tikzpicture}[scale=1.2]
\draw [line width=1.5pt](45:1) .. controls (45:.7) and (15:.7) .. (15:1);
\draw [line width=1.5pt](45:1)(165:1) .. controls (165:.7) and (135:.7) .. (135:1);
\draw [line width=1.5pt](45:1)(285:1) .. controls (285:.7) and (255:.7) .. (255:1);
\draw[line width=1.5pt, densely dashed] (345:1) .. controls (345:.7) and (315:.7) .. (315:1);
\draw [line width=1.5pt,densely dashed] (105:1) .. controls (105:.7) and (75:.7) .. (75:1);
\draw [line width=1.5pt, densely dashed] (225:1) .. controls (225:.7) and (195:.7) .. (195:1);
\foreach \x in {30,90,...,330}
	\draw[->] (\x:.75)--(\x:.5);
\end{tikzpicture}
\quad\quad
\begin{tikzpicture}[scale=1.2]
\draw [line width=1.5pt](45:1) .. controls (45:.5) and (135:.5) .. (135:1);
\draw [line width=1.5pt](165:1) .. controls (165:.5) and (255:.5) .. (255:1);
\draw [line width=1.5pt](285:1) .. controls (285:.5) and (15:.5) .. (15:1);
\draw[line width=1.5pt,densely dashed] (345:1) .. controls (345:.5) and (75:.5) .. (75:1);
\draw[line width=1.5pt,densely dashed] (105:1) .. controls (105:.5) and (195:.5) .. (195:1);
\draw [line width=1.5pt,densely dashed] (225:1) .. controls (225:.5) and (315:.5) .. (315:1);
\foreach \x in {30,90,...,330}
	\draw[->] (\x:.45)--(\x:.2);
\end{tikzpicture}
\quad\quad
\begin{tikzpicture}[scale=1.2]
\draw [line width=1.5pt](45:1) .. controls (45:.7) and (15:.7) .. (15:1);
\draw [line width=1.5pt](165:1) .. controls (165:.7) and (135:.7) .. (135:1);
\draw [line width=1.5pt](285:1) .. controls (285:.7) and (255:.7) .. (255:1);
\draw[line width=1.5pt,densely dashed] (345:1) .. controls (345:.7) and (315:.7) .. (315:1);
\draw [line width=1.5pt,densely dashed] (105:1) .. controls (105:.7) and (75:.7) .. (75:1);
\draw [line width=1.5pt,densely dashed] (225:1) .. controls (225:.7) and (195:.7) .. (195:1);
\foreach \x in {30,90,...,330}
	\draw[->] (\x:.75)--(\x:1);
\end{tikzpicture}

\caption{Cross-sectional views of (from left to right): a source, surgered source, sink, and surgered sink. The solid lines are portions of the intersections of a meridional disk of $U_i$ with an embedded family of annuli, and the dashed lines correspond to a separate family of annuli. By examining the co\"orientations in the leftmost picture, we see that none of the regions incident to $A$ correspond to sources, so that all of the incident regions correspond to subtori distinct from the source corresponding to $A$. Thus the surgery merges distinct subtori, reducing the total number of subtori; the argument is symmetric for sinks. In the general case, sources and sinks may have any even number of boundary subannuli.}
\label{source sink}
\end{figure}

Now define $S_\ell=\mr S_\ell\cup \left(\bigcup_i A_{i,\ell} \right) \cup \left( \bigcup_i D_{i,\ell} \right)$. This gives us a collection of closed surfaces $S_1,\dots, S_n$ embedded in $M$. We see from the definition of the puncturing map $P$ that $P([S_\ell])=\mr v_\ell$, and since $P$ is injective, it follows that $[S_\ell]=v_\ell$.

Each disk of $\bigcup_i D_{i,\ell}$ intersects $c$ positively because $\mr S_\ell$ is carried by $\mr \tau$. Therefore $\#(\bigcup_i D_{i,\ell})=i(v_\ell,c)$. Each annulus of $\bigcup_i A_{i,\ell}$ contributes 0 to $\chi_-(S_\ell)$. Therefore we have 
\begin{align}
\chi_-(S_\ell)&=\chi_-(\mr S_\ell)-i(v_\ell,c) \nonumber\\
&=x(\mr v_\ell) -i(v_\ell,c)\label{normminimizing1}\\
&=x(v_\ell) \label{normminimizing2},
\end{align}
where (\ref{normminimizing1}) follows from the fact that $\mr S_\ell$ is carried by $\mr\tau$ and is thus norm-minimizing, and (\ref{normminimizing2}) comes from Lemma \ref{norm and filling}.

Let $B_\sigma$ be the branched sum of $S_1,\dots,S_n$. As the branched sum of norm-minimizing surfaces, it will be almost taut by Lemma \ref{branched sum almost taut} provided it carries no spheres (we need not consider disks since $B_\sigma$ has no boundary).  It spans $\sigma$ because it carries a representative of each vertex. 

Let us briefly review the structure of $B_\sigma$. Inside $\mr M$, $B_\sigma$ is a branched surface lying in a regular neighborhood of $B_{\mr\tau}$. Its boundary is a train track lying in a regular neighborhood of $\gamma$. Inside each $U_i$, $B_\sigma$ is a branched sum of meridional disks with at most 2 embedded collections of annuli whose boundaries have ladderpole slope on $T_i$. We will use the notations $A_{i,\ell}$ and $D_{i,\ell}$ to denote the images of the $A_{i,\ell}$ and $D_{i,\ell}$ under the branched sum isotopy.

We now show that $B_\sigma$ is a homology branched surface, which will show that $B_\sigma$ carries no spheres, whence it is almost taut by Lemma \ref{tautness from carrying}. Since an almost taut homology branched surface in a pseudo-Anosov mapping torus is taut by Lemma \ref{almost taut to taut}, this will complete the proof. 

If $p$ is any point in $B_\sigma$ outside of $U$, there is a closed positive transversal through $p$ because $B_{\mr\tau}$ is a homology branched surface. 
Now suppose $p\in B_\sigma$ lies inside $U_i$ for some $i$. We construct a path $f$ from $p$ to the interior of $\mr M$ which is a positive transversal to $B_\sigma$.

Begin the path $f$ at $p$ by traveling from $B_\sigma$ into $U\setminus B_\sigma$ in the direction of the co\"orientation of $B_\sigma$. The endpoint of $f$ lies in a component $C$ of $\closure(U_i)\setminus (\bigcup_\ell D_{i,\ell})$ that is homeomorphic to a solid cylinder $\{x\in \mathbb R^2\mid ||x||\le 1\}\times (0,1)$. The annuli $\bigcup_\ell A_{i,\ell}$ cut $C$ into smaller subcylinders whose sides are either co\"oriented portions of $\bigcup_\ell A_{i,\ell}$ or subsets of $\partial U_i$.

Because there are no sinks in $U_i$, this subcylinder is either adjacent to $\partial U_i$ or possesses an outwardly co\"oriented wall. If the subcylinder is adjacent to $\partial U_i$, extend $f$ to a point $q$ outside of $U_i$. Otherwise, extend $f$ through an outwardly co\"oriented wall to enter a new subcylinder of $C$, and iterate this procedure. Each component of $\left(\bigcup_\ell A_{i,\ell}^1\right )\cap C$ is a co\"oriented rectangle separating $C$ into two components. Each time $f$ passes through one of these rectangles, $f$ is blocked from passing through a second time because of the rectangle's co\"orientation. It follows that $f$ never returns to the same subcylinder of $C$. Thus the procedure eventually terminates, and $f$ may be extended to $q\notin U_i$.

A symmetric construction using the fact that there are no sources in $U_i$ shows that there is a negative transversal $h^{-1}$ from $p$ to a point $r$ exterior to $U_i$. Using the fact that $B_\tau$ is transverse to the pseudo-Anosov suspension flow in $M\setminus c$, we can find a positive transversal $g$ from $q$ to $r$. The concatenation $fgh$ is then a closed positive transversal through $p$.  
\end{proof}

We remark that Theorem \ref{THBS from cusp condition} extends Theorem 1.5 in \cite{Mo}, which states that if every vertex of the fibered face $\sigma$ has positive intersection with each singular orbit of $\varphi$ (and in particular is non-ladderpole), then $M$ has a taut oriented branched surface.

\begin{theorem}\label{small betti}
If $b_2(M)\le 3$, then each boundary torus of $\mathring{M}$ witnesses at most two ladderpole vertex classes.
\end{theorem}

\begin{proof}
The oriented sum of two surfaces which are ladderpole at a component $T_i$ of $T$ is again ladderpole at $T_i$. Therefore
the same is true for homology classes, and in particular the sum of two ladderpole classes lies in the boundary of $\sigma$ by Lemma \ref{types of boundaries}. We conclude that all vertex classes which are ladderpole at the same boundary component of $\mathring{M}$ lie in the same facet of $\sigma$. The dimension of $\sigma$ is at most 2 by assumption, so this facet has dimension at most 1. Since a 1-cell has two boundary points, at most two vertex classes can be ladderpole at the same $T_i$.
\end{proof}

\begin{corollary}\label{main corollary}
If $b_2(M)\le 3$, then $M$ has a taut homology branched surface spanning $\sigma$.
\end{corollary}

The following corollary was known to Mosher in \cite{Mo}, but we include it here because it follows very easily from our results.

\begin{corollary}\label{oneorbitcor}
If $\phi$ has only one singular orbit, then $M$ has a taut homology branched surface spanning $\sigma$.
\end{corollary}

\begin{proof}
In this case there are no ladderpole classes in $H_2(M)$. Indeed, if $\alpha\in H_2(M)$, then a representative of $\mr\alpha$ carried by $B_{\mr \tau}$ cannot have all ladderpole boundary components by Lemma \ref{not all exceptional}. Since $\mr M$ has only one boundary component, such a representative of $\mr\alpha$ cannot have \emph{any} ladderpole boundary components. The result follows from Theorem \ref{THBS from cusp condition}.
\end{proof}

\textsc{Yale University}

\textit{Email address}: \texttt{michael.landry@yale.edu}

\end{document}